\documentclass[a4paper,11pt]{article}
\usepackage{authblk,setspace}

\usepackage{amsmath} 
\usepackage{amsthm} 
\usepackage{amssymb}	
\usepackage{multicol} 
\usepackage{soul} 
\usepackage{cancel}
\usepackage{dsfont,bbm}
\usepackage{empheq,cases}
\usepackage[export]{adjustbox}

\DeclareFontFamily{OT1}{pzc}{}
\DeclareFontShape{OT1}{pzc}{m}{it}{<-> s * [1.10] pzcmi7t}{}
\DeclareMathAlphabet{\mathpzc}{OT1}{pzc}{m}{it}

\usepackage[british]{babel}
\usepackage{csquotes}
\usepackage{calligra}
\DeclareMathAlphabet{\mathcalligra}{T1}{calligra}{m}{n}

\usepackage[low-sup]{subdepth}
\usepackage{hyperref}

\let\baraccent=\= 
\renewcommand{\=}[1]{\stackrel{#1}{=}} 

\newcommand{\norm}[1]{\left\| #1 \right\|}

\newcommand{\inv}{^{-1}}

\newcommand{\red}[1]{\textcolor{black}{#1}}

\usepackage{graphicx}
\usepackage{bm,color,comment}
\usepackage{mathrsfs,enumerate}
\usepackage{subfig}
\usepackage{caption}
\usepackage[font=small,format=plain,labelfont=bf,textfont=it]{caption}
\usepackage{url}
\numberwithin{equation}{section}
\definecolor{refkey}{rgb}{0.9451,0.2706,0.4941}
\definecolor{labelkey}{rgb}{0.9451,0.2706,0.4941}

\theoremstyle{plain}
\newtheorem{theorem}{Theorem}[section]
\newtheorem*{theorem*}{Theorem}
\newtheorem{corollary}[theorem]{Corollary}
\newtheorem{lemma}[theorem]{Lemma}

\newtheorem{proposition}[theorem]{Proposition}
\theoremstyle{definition}

\newtheoremstyle{myremstyle}
  {\topsep} 
  {1.4\topsep} 
  {} 
  {} 
  {\itshape} 
  {.} 
  {.5em} 
  {} 

\theoremstyle{myremstyle} \newtheorem{remark}{Remark}

\usepackage{fancybox,here,mathtools}

\newcommand*{\from}{\colon}
\newcommand*{\bigmid}{\:\big|\:}

\newcommand*{\Biggmid}{\:\Bigg|\:}

\newcommand{\indic}{\mathbb{I}}

\newcommand{\bbR}{{\mathbb{R}}}
\newcommand{\bbN}{{\mathbb{N}}}

\newcommand{\calK}{{\mathscr{K}}}
\newcommand{\calS}{{\mathscr{S}}}

\newcommand{\e}{\mathrm{e}}

\newcommand{\bbE}{\mathbb{E}}

\newcommand{\dr}{\ensuremath{\,{\rm d}r}}
\newcommand{\ds}{\ensuremath{\,{\rm d}s}}

\newcommand{\innprod}[2]{\left\langle #1, #2 \right\rangle}

\title{Discrete maximal regularity of an implicit Euler--Maruyama scheme with non-uniform time discretisation for a class of stochastic partial differential equations}
\usepackage{doi}

\theoremstyle{myremstyle} \newtheorem{assumption}{Assumption}

\newcommand{\what}{\widehat}
\newcommand{\TMAX}{1}

\newcommand{\dstoconv}[1]{\big[R^J \diamond #1 \big]^L}
\newcommand{\dsconvjk}[1]{\big[R^J \diamond #1 \big]^L_{\red{j}}}
\newcommand{\sumlall}{\red{\sum_{\ell=1}^{\infty}}}
\newcommand{\sumltrunc}{\red{\sum_{\ell=1}^{L}}}
\newcommand{\sumjall}{\red{\sum_{j=1}^{\infty}}}
\newcommand{\sumjtrunc}{\red{\sum_{j=1}^{J}}}

\newcommand{\HS}{\mathcal{L}_2}

\newcommand{\dW}{\mathrm{d}W}

\newcommand{\Addresses}{{
	\bigskip
  \footnotesize
  \textsc{School of Mathematics and Statistics}\par\nopagebreak
  \textsc{University of New South Wales}\par\nopagebreak
  \textsc{Sydney NSW 2052, Australia}\vspace{1pt}\par\nopagebreak
  \textit{E-mail}: \texttt{\href{mailto:y.kazashi@unsw.edu.au}{y.kazashi@unsw.edu.au}}
}}

\author{Yoshihito Kazashi}
\begin{document}
\maketitle
\begin{abstract}
An implicit Euler--Maruyama method with non-uniform step-size applied to a class of stochastic partial differential equations is studied.
A spectral method is used for the spatial discretization and the truncation of the Wiener process. 
 A discrete analogue of maximal $L^2$-regularity of the scheme and the discretised stochastic convolution is established, which has the same form as their continuous counterpart. 
\end{abstract}
\section{Introduction}
Our interest in this paper lies in a discrete analogue of maximal regularity 
for a class of stochastic partial differential equations (SPDEs) of parabolic type. In more detail, with a positive self-adjoint generator $-A$ with compact inverse densely defined on a separable Hilbert space $H$, we consider the equation
\begin{align}\label{eq:main}
\left\{
\begin{array}{rl}
\mathrm{d}X(t) &= {A} X(t) \mathrm{d}t + B(t,X(t)) \mathrm{d}W(t),\qquad\text{ for }\ t \in (0,1] \\
X(0) &= \xi,
\end{array}\right.
\end{align}
where the mild solution $X$ takes values in $H$. The assumption on $B$ and the $Q$-Wiener process $W$ will be discussed later. 
{The aim of this paper is to show a property of a prototypical discretisation to simulate the solution of such equations:   
	we show a discrete analogue of an estimate called maximal regularity (Corollary \ref{cor:max estim}).}

Maximal regularity is a fundamental concept in the theory of deterministic partial differential equations (see, for example \cite{Amann.H_1995_book_vol1,Kunstmann.P_Weis_2004_lecture_note,Lunardi.A_1995_book} and references therein). Similarly, in the study of stochastic partial differential equations, the maximal regularity is an important analysis tool \cite{DaPrato.D_2014_book,DaPrato.G_1982_reg_convo} as well as an active research area  \cite{vanNeerven.J_2012_stomaximal,vanNeerven.J_etal_2012_sto_meaximal_evol,Auscher.P_etal_2014_conical,vanNeerven.J_2015_Rbdd}. In our setting, the above equation~\eqref{eq:main} can be shown to satisfy the maximal regularity estimate of the form
\begin{align}
\int_0^1 \bbE
[
\|X(s)\|_{D(A^{\iota+\frac12})}^2
]\ds 
\le 
\|\xi\|_{D(A^{\iota+\frac12})}^2
+ \int_0^1 \bbE \big\|B(r,X(r))\big\|_{\HS(H_0,D(A^{\iota}))}^2
\dr,\label{eq:maximal reg}
\end{align}
where $\iota\ge 0$ is a suitable parameter depending on the operator $B$,  $D(A^{\iota+\frac12})$ is the domain of the fractional power $A^{\iota+\frac12}$ of $A$ in $H$, and $\HS(H_0,D(A^{\iota}))$ is the space of Hilbert--Schmidt operator from $H_0$, the Cameron--Martin space associated with $Q$, to $D(A^{\iota})$. More details will be discussed later.

In recent years, the study of discrete analogues of the maximal regularity has been attracting attention for deterministic partial differential equations \cite{Agarwal.R_etal_2014,Ashyralyev.A_1994_book,Blunck.S_2001_maximal,Kemmochi.T_2016_disc_maximal,Kemmochi.T_Saito_2016_arXiv,Kovacs.B_etal_2015_arxiv,Leykekhman.D_etal_2015_arxiv}; {to the best of the author's knowledge, corresponding properties of numerical methods for stochastic PDEs have not been addressed in the literature.}

{Maximal regularity of stochastic and deterministic equations are different in nature. As we will see in \eqref{eq:maximal reg}, given a suitable smoothness of the initial data, the solution is ``one-half spatially smoother'' than the range of the diffusion operator $B(t,x)$. This estimate \red{is} optimal, in that the solution cannot be spatially smoother in general (see \cite[Example 5.3]{Kruse.R_Larsson_2012_optimal_reg}). To put it another way, as described in \cite[Chapter 6]{DaPrato.D_2014_book}, the regularity one can obtain is the half of the corresponding regularity for the deterministic case.}

{We focus on the case where the operator $A$ and the covariance operator $Q$ share the same eigensystems. This prototypical setting is partly motivated by applications in environmental modelling and astrophysics, where  
	covariance operators---of the random fields \cite{Baldi.P_etal_2007_homogeneous,Marinucci.D_2011_book}, and of the Wiener process for the stochastic heat equations \cite{Lang.A_Schwab_2015,Anh.V.V_etal_2018_fractional}, for example---the eigenspaces of which are the same as those of the Laplace operators play important roles. 
	In simulations, it is desirable that discretisations users employ inherit properties of the solution of the model considered. Our results show the method we consider preserves a spatial regularity---maximal regularity---in a suitable sense.  
} 

{As a spatial discretisation we consider the spectral-Galerkin method. 
	The Wiener process, which is assumed to admit a series representation, takes its value in an infinite-dimensional space. In practice, we can simulate only finitely many of them. We approximate the Wiener process by truncation, i.e., we use a type of truncated Karhunen--Lo\`{e}ve approximation.}

{Temporally, we consider the implicit Euler--Maruyama method with a non-uniform time discretisation. The aforementioned approximation of the Wiener process introduces one-dimensional Wiener processes multiplied by varying scalars---the eigenvalues of the covariance operator.  
	Motivated by this observation, we allow the discretisation of each Wiener process to depend on these scalars. 
	The algorithm we consider is first proposed by M\"{u}ller-Gronbach and Ritter~\cite{Muller-Gronbach.T_Ritter_2007_FoCM_lower_bd,Muller-Gronbach.T_Ritter_2007_BIT_multiplicative}, for the stochastic heat equation on the unit cube. In \cite{Muller-Gronbach.T_Ritter_2007_FoCM_lower_bd,Muller-Gronbach.T_Ritter_2007_BIT_multiplicative}, the resulting non-uniform scheme was shown to achieve an asymptotic optimality under a suitable step size, which in general cannot be achieved by schemes with uniform step-size. }

{
	The results we establish show that the non-uniform discretisation still preserves maximal regularity in a suitable sense. 
	The algorithm we consider includes the implicit Euler--Maruyama method with the uniform time discretisation as a special case---the case where one uses the same step size for all one-dimensional Wiener processes---even though we, in general, lose the aforementioned optimality. As a consequence, we obtain a discrete analogue of maximal regularity for the standard implicit Euler--Maruyama method: the discretisation with the uniform step size.}

The structure of this paper is as follows. 
Section~\ref{maximal-sec:setting} recalls some definitions and basic results needed in this paper. Section~\ref{maximal-sec:discrete} introduces the discretised scheme we consider. Then, in Section~\ref{maximal-sec:disc reg estim} we show a discrete maximal regularity. 
Then, we conclude this paper in Section~\ref{maximal-sec:conclusion}. 
\section{Setting}\label{maximal-sec:setting}
By $H$ we denote a separable $\bbR$-Hilbert space $(H,\langle\cdot,\cdot\rangle,\|\cdot\|)$. Let $-A:D(A)\subset H\to H$ be a self-adjoint, positive definite linear operator that is densely defined on $H$, with compact inverse. Then, $A$ is the generator of the $C_0$-semigroup $(S(t))_{t \ge 0}:=(\e^{At})_{t\ge 0}$ acting on $H$ that is analytic. Further, there exists a complete orthonormal system $\{h_{{j}}\}$ for $H$ such that $-A h_{{j}} = \lambda_{j} h_{{j}}$, each eigenspace is of finite dimensional, and 
\[
0< \lambda_{1}< \lambda_{2}<\dotsb< \lambda_{j}<\dotsb,
\]
and $\lambda_{j}\to\infty$ as $j\to\infty$ unless the compact inverse $-A^{-1}$ is finite rank.
{For simplicity, we assume the dimension of each eigenspace is $1$.} 
Then, we have the spectral representation 
\[
S(t) x = \sumjall \e^{-\lambda_{j} t}\innprod{x}{h_{{j}}}h_{{j}}\in H.
\]

For $r\in\bbR$, let us define the domain $D(A^r)$ of the fractional power $A^r$ of $A$
by
\begin{align*}
D(A^{r}) := \Bigg\{ x \in H \Biggmid \|x\|^2_{D(A^{r})} =\sumjall \lambda_{j}^{2{r}}
\innprod{x}{h_{{j}}}^2 < \infty \Bigg\}.
\end{align*}
{We obtain a separable Hilbert space $(D(A^{r}),\langle\cdot,\cdot\rangle_{D(A^{r})},\|\cdot\|_{D(A^{r})})$ by setting $\langle\cdot,\cdot\rangle_{D(A^{r})}:=\langle A^r\cdot,A^r\cdot\rangle$.} 

{For more details for the set up above}, see for example  \cite{Kato.T_book_1995_reprint,Lunardi.A_1995_book,Sell.G_You_2013_book,Yosida.K_book_1995reprint}.

{Let $(\Omega,\mathscr{F},\mathbb{P})$ be a probability space equipped with a filtration satisfying the usual conditions.} 
By $W\from[0,1]\times\Omega\to H$ we denote the $Q$-Wiener process with a covariance operator $Q$ of the trace class.  We assume that the Wiener process $W$ is adapted to the filtration. 
Further, we assume that the eigenfunctions $h_{{\ell}}$ of $A$ is also eigenfunctions of $Q$ with
\begin{align*}
Q h_{{\ell}}=  q_{\ell} h_{{\ell}},
\end{align*}
such that
$
\mathrm{Tr}(Q)
=\sumlall \innprod{Q h_{{\ell}}}{h_{{\ell}}}
=\sumlall  q_{\ell}<\infty
$. 
It is well-known that $W$ taking values in $H$ can be characterised as 
\[W(t)
=
\sumlall
\sqrt{ q_{\ell}}
\beta_{{\ell}}(t)
h_{{\ell}}\quad\text{
	a.s.},\]
where $\beta_{{\ell}}$ are independent one-dimensional standard Brownian motions with the zero initial condition realised on {$(\Omega,\mathscr{F},\mathbb{P})$}  that are adapted to the underlying filtration{, and that the series converges in the Bochner space $L^2(\Omega;C([0,1];H))$}.  
The $Q$-Wiener process takes values in $H$ by construction. Here, since $A$ and $Q$ are assumed to share the same eigenfunctions, we can provide finer characterisations of the regularity.
\remark{
	Let $r\ge0$ and $t\in(0,1]$. Then, 
	$
	\sumlall
	\lambda_{\ell}^{2r}{ q_{\ell}}<\infty
	$
	if and only if 
	$
	W(t)
	\in D(A^{r})$,
	a.s. 
	Indeed, we have
	$
	\mathbb{E}[
	\norm{W(t)}_{D(A^{r})}^2
	]
	=t
	\sumlall \lambda_{\ell}^{2 r} q_{\ell}
	$.
}

We introduce the Hilbert space
$
H_0 = Q^{1/2}(H) 
$
equipped with the inner product
\[
\innprod{ h_1}{ h_2}_0 = \langle Q^{-1/2} h_1,Q^{-1/2} h_2\rangle \quad \mbox{ for } h_1, h_2 \in H,
\] 
where $Q^{-1/2}:=(Q^{1/2}|_{(\ker(Q^{1/2}))^\perp})\inv\from {H_0} \to (\ker(Q^{1/2}))^\perp$ is the pseudo-inverse of $Q^{1/2}$. 

In the following, $a \preceq b$ means that $a$ can be bounded by some constant times $b$ uniformly with respect to any parameters on which $a$ and $b$ may depend. 
Throughout this paper, we assume the following. 
\begin{assumption}\label{assump:cond on B}
	We assume $B\from[0,1]\times \red{H}\to \HS(H_0, \red{H})$ is $\mathcal{B}([0,1])\otimes\mathcal{B}(\red{H})/\mathcal{B}(\HS(H_0, \red{H}))$-measurable, where for a given normed space $(\mathscr{X},\|\cdot\|_{\mathscr{X}})$ the Borel $\sigma$-algebra associated with the norm topology is denoted by $\mathcal{B}(\mathscr{X})$. 
	Further, let $B$ satisfy
	\begin{equation}\label{eq:lip-H}
	\| B(t,u) - B(t,v)\|_{\HS(H_0, \red{H})}
	\preceq
	\|u-v\|,
	\quad 
	\text{ for $\ t\in [0,1]$, $u,v\in \red{H}$.}
	\end{equation}
	\red{Moreover, let $\iota\in[0,1/2]$ be given. We assume for any $t \in [0,1]$, $u \in {D(A^{\iota})}$ we have $B(t,u)\in \HS(H_0, D(A^{\iota}))$ and}
	\begin{equation}\label{eq:lin-DA}
	\|B(t,u)\|_{\HS(H_0, D(A^{\iota}))} 
	\preceq
	1+ \|u\|_{ D(A^{\iota})}.
	\end{equation}
\end{assumption}
\red{
	The condition \eqref{eq:lin-DA} implies $\sup_{t\in [0,1]}\|B(t,0)\|_{\HS(H_0, D(A^{\iota}))} 
	\preceq 1$. Thus, together with \eqref{eq:lip-H} we see that Assumption~\ref{assump:cond on B} implies 
	\begin{equation}\label{eq:lin-H}
	\|B(t,u)\|_{\HS(H_0, H)}
	\preceq
	C_{\iota}(1+ \|u\|)<\infty,
	\end{equation} 
	for $t\in[0,1]$, $u\in H$, with a constant $C_{\iota}>0$.
}

We recall the following existence result, which can be found in, for example, \cite{DaPrato.D_2014_book}.
\begin{theorem}
	Suppose that the mapping $B$ satisfies Assumption~\ref{assump:cond on B} with some $\iota\ge0$. Then, for $\xi\in \red{H}$ there exists \red{an} $\red{H}$-valued continuous process
	$(X(t))_{t\in [0,1]}$ adapted to the underlying filtration satisfying the usual conditions such that
	\begin{equation}
	X(t) = S(t) \xi + \int_0^t S(t-s) B(s,X(s)) \mathrm{d}W(s), 
	\quad t \in [0,1] \quad \text{a.s.}
	\label{eq:mild-sol}
	\end{equation}
	Moreover, this process is uniquely determined a.s., and it is called
	the mild solution of~\eqref{eq:main}. 
	Further, {for any $p\ge2$ we have}
	\begin{align}\label{eq:sol sup E Hr bd}
	\sup_{ t \in [0,1] } \bbE\|X(t)\|_{ }^p < \infty.
	\end{align}
\end{theorem}
For more details, see for example \cite[Sec.~7.1]{DaPrato.D_2014_book}. 
For the mild solution $X$, let 
\begin{equation*}
X(t) = \sumjall X_{{j}}(t) h_{{j}}, 
\quad X_{{j}}(t) = 
\innprod{X(t)}{h_{{j}}}.\label{eq:defX}
\end{equation*}
Then, the processes $X_{{j}} = (X_{{j}}(t))_{t \in [0,1]}$
satisfy the following bi-inifinite system of stochastic
differential equations:
\begin{align*}
\left\{
\begin{array}{rl}
\mathrm{d}X_{{j}}(t) &= -\lambda_{j} X_{{j}}(t) \mathrm{d}t 
+
\sumlall
\sqrt{ q_{\ell}}\innprod{B(t,X(t)) h_{{\ell}}}{ h_{{j}}} 
\mathrm{d}\beta_{{\ell}}(t) \\
X_{{j}}(0) &= \innprod{\xi}{h_{{j}}} ,
\qquad \text{ for }\quad j \in \bbN.
\end{array}\right.
\end{align*}
Each process $X_{{j}}$ is {given as}
\begin{align*}
\begin{split}
X_{{j}}(t) = &\e^{-\lambda_{j} t} \innprod{\xi}{h_{{j}}} \\
&+
\sumlall   
\sqrt{ q_{\ell}}\int_{0}^t \e^{-\lambda_{j}(t-s)}
\innprod{B(s,X(s)) h_{{\ell}}}{ h_{{j}}} 
\mathrm{d}\beta_{{\ell}}(s),
\end{split}
\end{align*}
{where the series in the second term is convergent in $L^2(\Omega)$, due to \eqref{eq:sol sup E Hr bd} and Assumption \ref{assump:cond on B}.}

We have the following spatial regularity result.
\begin{proposition}
	Suppose that Assumption~\ref{assump:cond on B} is satisfied with some $\red{\iota\in[0,1/2]}$, and that the initial condition satisfies $\xi\in  D(A^{\iota})$. 
	Then, we have the estimate
	\begin{align}
	\int_0^1 \bbE
	\|X(s)\|_{ D(A^{\iota+1/2})}^2
	\ds 
	\le 
	\|\xi\|_{ D(A^{\iota})}^2
	+ \int_0^1 \bbE \big\|B(r,X(r))\big\|_{\HS(H_0, D(A^{\iota}))}^2
	\dr.\label{eq:maximal reg shown}
	\end{align}
\end{proposition}
\begin{proof}
	It\^{o}'s isometry yields
	\begin{align*}
	\lambda_j^{2\iota+1}\bbE(X_{{j}}(s))^2 &=
	\exp(-2 \lambda_{j} s) \lambda_j^{2\iota+1}\innprod{\xi}{h_{{j}}}^2 \notag\\
	& +\! \int_0^s\! \exp(-2\lambda_{j} (s-r)) 
	\lambda_j
	\bbE \| B^*(r,X(r)) \lambda_j^{\iota+\frac12} h_{{j}}\|^2_{H_0} \mathrm{d}r,
	\end{align*}
	where $B^*(r,X(r))$ denotes the adjoint operator of $B(r,X(r))$. 
	Therefore, it holds that 
	\[
	\int_0^1 \bbE
	[
	\lambda_j^{2\iota+1}|X_{{j}}(s)|^2 
	]\ds 
	\le \lambda_j^{2\iota}\innprod{\xi}{h_{{j}}}^2
	+ \int_0^1 \bbE \|B^*(r,X(r)) \lambda_j^{2\iota} h_{{j}}\|^2_{H_0} \dr
	,\] and thus summing over $j\ge1$ 
	yields the desired result.
	
	\red{We note that for $\iota\in[0,1/2]$ the right hand side of \eqref{eq:maximal reg shown} is finite. To see this, we first note that \eqref{eq:lin-H} together with \eqref{eq:sol sup E Hr bd} implies
		\begin{equation*}
		\int_0^1 \bbE
		\|X(s)\|_{ D(A^{1/2})}^2
		\ds 
		\le 
		\|\xi\|_{ }^2
		+ \int_0^1 \bbE \big\|B(r,X(r))\big\|_{\HS(H_0, H)}^2
		\dr<\infty.
		\end{equation*}
		Thus, from \eqref{eq:lin-DA} we have 
		\begin{align*}
		\int_0^1\! \bbE \|B(r,X(r))\|_{\HS(H_0, D(A^{\iota}))}^2\dr
		&\preceq 1 + \int_0^1\! \bbE \|X(r)\|_{D(A^{\iota})}^2 \dr\\
		&\le c_{\iota}\Big(1+ \int_0^1\! \bbE
		\|X(r)\|_{ D(A^{1/2})}^2\dr\Big)
		<\infty,
		\end{align*}
		for some constant $c_{\iota}>0$. 
	}
\end{proof}
\remark{
	We note that the solution is spatially one half smoother than the range of $B(t,x)$. This is in general optimal, in that the solution cannot be spatially smoother in general (\cite[Example 5.3]{Kruse.R_Larsson_2012_optimal_reg}). 
	For more details, see  \cite{Kruse.R_Larsson_2012_optimal_reg,Kruse.R_2014_lecture_note} and references therein. For recent developments of maximal regularity theory, see \cite{vanNeerven.J_2012_stomaximal,vanNeerven.J_etal_2012_sto_meaximal_evol}. 
}
\section{Discretisation}\label{maximal-sec:discrete}
This section introduces the scheme proposed by M\"{u}ller-Gronbach and Ritter~\cite{Muller-Gronbach.T_Ritter_2007_FoCM_lower_bd,Muller-Gronbach.T_Ritter_2007_BIT_multiplicative}. In this regard, let us first discretise the interval $[0,1]$ with a uniform partition,
i.e., we partition the interval with $t_{i} = i/n$, for 
$i=0,1,2,\dotsc,n$. 
{For integers $J,L\in\bbN$, an It\^{o}--Galerkin approximation $\overline{X}(t_{i})$ to \eqref{eq:mild-sol} with the temporal discretisation being the implicit Euler--Maruyama scheme with a uniform time discretisation is given by}
\begin{align}
{\overline{X}^{J,L}(t_{i})
	=\sumjtrunc
	\overline{X}^{J,L}_{{j}}(t_{i}) h_{{j}},\ \text{ for }\ i=0,\dots,N},
\label{eq::ch-maximal-unif}
\end{align}
{with coefficients $\langle {\overline{X}(t_{i})},{h_{{j}}}\rangle$ defined by $
	\overline{X}^{J,L}_{{j}}(0) = \innprod{\xi}{h_{{j}}}
	$, and }
\begin{align*}
{\overline{X}}^{J,L}_{{j}}({t_{i}})
= \Big(1+\frac{\lambda_{j}}{n}\Big{)}^{-1}&\bigg({\overline{X}}^{J,L}_{{j}}{(t_{i-1})}
\nonumber\\
&+\sumltrunc
\sqrt{ q_{\ell}}\innprod{B(t_{i-1},{\overline{X}}^{J,L}{(t_{i-1})}) h_{{\ell}}}{ h_{{j}}}
( \beta_{{\ell}}{(t_{i})} - \beta_{{\ell}}{(t_{i-1})})\bigg).
\end{align*}

M\"{u}ller-Gronbach and Ritter~\cite{Muller-Gronbach.T_Ritter_2007_FoCM_lower_bd,Muller-Gronbach.T_Ritter_2007_BIT_multiplicative} noted that the projected $Q$-Wiener processes $\sqrt{ q_{\ell}}\beta_{{\ell}}=\sqrt{\innprod{Q h_{{\ell}}}{h_{{\ell}}} }\beta_{{\ell}}=\innprod{W(t)}{h_{{\ell}}}$ have varying variances depending on the index $\ell$. 
This observation motivated them to use different step-sizes depending on $\ell$. Following them, we evaluate the standard one-dimensional Wiener process $\beta_{{\ell}}$ at each level ${\ell}={1},\dots,L$ at the corresponding $n_{\ell}\in\bbN$ nodes
\[
0<t_{1,\ell}<\cdots<t_{n_{\ell},\ell}=\TMAX,
\quad\text{ where }\ 
t_{i,\ell} = \frac{i}{n_{\ell}} \quad\text{for }\  i=0,\ldots,n_{\ell}.
\]
Then, the discretisation of the truncated $Q$-Wiener process 
$\displaystyle\sumltrunc\sqrt{ q_{\ell}}\beta_{{\ell}} h_{{\ell}}$ in general results in a non-uniform time discretisation:
\begin{align*}
0 =: \tau_0 < \dots < \tau_N := \TMAX,\quad\text{ where }\quad\{\tau_0,\dotsc,\tau_{N} \} 
:=
\bigcup_{\ell=1}^{L} \{ t_{0,\ell},\ldots,t_{n_{\ell},\ell} \},
\end{align*}
and $t_{0,\ell}=\tau_0=0$ for all $\ell\in\bbN$.
To write our scheme in the recursive form, we introduce the following notations. Let
\begin{align*}
\calK_{\eta} 
:=
\big\{ \ell \in \{0,1,\dotsc,L\} \bigmid 
\tau_{\eta} \in \{t_{0,\ell},\ldots,t_{n_{\ell},\ell}  \} \big\},
\end{align*}
for $\eta=0,\dotsc,N$ and we define $s_{\eta,\ell}$ for $\eta=1,\dotsc,N$ and $\ell=1,\dotsc,L$ 
by
\[
s_{\eta,\ell} := \max \big\{\{t_{0,\ell},\dotsc,t_{n_{\ell},\ell} \}\cap [0,\tau_{\eta}) \big\}.
\]
We further introduce the following notation for the product of eigenvalues of the operator $(I-\frac 1{\tau_\nu - \tau_{\nu-1}}A)\inv$, which we use for the approximation of the semigroup generated by $A$. For any $\tau_{\eta_1}\le \tau_{\eta_2}$, we let
\begin{align}\label{def:Gamma}
\mathfrak{R}_{j}(\tau_{\eta_1},\tau_{\eta_2})
:=
{\prod_{\nu=\eta_1+1}^{\eta_2} \frac{1}{1 + \lambda_{j}(\tau_\nu - \tau_{\nu-1})},}
\end{align}
{with the convention $\prod_{\emptyset}=1$.} 
Note that $s_{\eta,\ell},t_{i-1,\ell}\in\{\tau_{1},\dots,\tau_{N}\}$.
Then, for $\eta=1,\dotsc,N$, the drift-implicit Euler--Maruyama scheme in the recursive form is given by,
\begin{equation} 
\begin{aligned}
\what{X}^{J,L}_{{j}}(\tau_{\eta}) = 
\mathfrak{R}_{j}(\tau_{\eta-1},\tau_{\eta})
\Bigg(
\what{X}^{J,L}_{{j}}(\tau_{\eta-1})
+
&{\sum_{\ell \in \calK_{\eta}}}
\sqrt{ q_{\ell}}\innprod{B(s_{\eta,\ell}, \what{X}^{J,L}(s_{\eta,\ell}) ) h_{{\ell}}}{h_{{j}}}\\
&\times
\mathfrak{R}_{j}(s_{\eta,\ell},\tau_{\eta-1})
( \beta_{{\ell}}(\tau_{\eta}) - \beta_{{\ell}}(s_{\eta,\ell}))
\Bigg).
\end{aligned}
\label{eq:def-Xjk-rec}
\end{equation}
Equivalently, the above can be written in the convolution form
\begin{align}
\what{X}^{J,L}_{{j}}(\tau_{\eta})
&=
\mathfrak{R}_{j}(\tau_0,\tau_{\eta}) \innprod{\xi}{h_{{j}}}
+
\sumltrunc
\sum_{\tau_1\le t_{i,\ell} \le \tau_{\eta}}
\sqrt{ q_{\ell}}\innprod{B(t_{i-1,\ell}, \what{X}^{J,L}(t_{i-1,\ell})) h_{{\ell}}}{ h_{{j}}}
\nonumber\\
&\phantom{\mathfrak{R}_{j}(\tau_0,\tau_{\eta}) \innprod{\xi}{h_{{j}}}}\quad\quad\times
\mathfrak{R}_{j}(t_{i-1,\ell},\tau_{\eta})
( \beta_{{\ell}}(t_{i,\ell}) - \beta_{{\ell}}(t_{i-1,\ell})).
\label{eq:Euler conv}
\end{align}
Then, we use 
\begin{align}
\what{X}^{J,L}(\tau_{\eta})
=
\sumjtrunc
\what{X}^{J,L}_{{j}}(\tau_{\eta}) h_{{j}}
\label{eq:full-scheme}
\end{align}
for $\eta=1,\dots,N$ as our approximate solution.

{We note that this scheme generalises the aforementioned approximation $\overline{X}^{J,L}$ with the uniform time step as in \eqref{eq::ch-maximal-unif}:  
	$\overline{X}^{J,L}$ is nothing but $\what{X}^{J,L}$ with 
	$n_{\ell}=N$ for $\ell=1,\dots,L$.
}

\section{Discrete regularity estimate}\label{maximal-sec:disc reg estim}
{First, let} $\mathscr{P}_{J}x := \sumjtrunc \innprod{x}{h_{{j}}} h_{{j}}$ for $x\in H$. {Further, by writing $\prod_{\emptyset}=I$ we let}
\begin{align}
{{R}(\tau_{\eta_1},\tau_{\eta_2};A):=
	\prod_{\nu=\eta_1+1}^{\eta_2} \Big(I-\frac{1}{\tau_{\nu}-\tau_{\nu-1}}A\Big)^{-1},
}
\end{align}
{where the meaning of the product symbol is unambiguous due to the commutativity of $(I-\frac{1}{\tau_{\nu}-\tau_{\nu-1}}A)^{-1}$'s.}

For $j\in\{1,\dotsc,J\}$ 
and $\eta\in\{1,\dotsc,N\}$, define 
\begin{align}
\begin{split}
&
\dstoconv{ B(\cdot, \what{X}^{J,L}(\cdot)) }(\tau_{\eta}) \\
&:=
\sumltrunc \sum_{\tau_1\le t_{i,\ell} \le \tau_{\eta}}
\mathscr{P}_J {R}(t_{i-1,\ell},\tau_{\eta};A)
B(t_{i-1,\ell}, \what{X}^{J,L}(t_{i-1,\ell})) 
\sqrt{ q_{\ell}} h_{{\ell}}
(\beta_{{\ell}}(t_{i,\ell}) - \beta_{{\ell}}(t_{i-1,\ell})).
\end{split}
\label{def:disc conv} 
\end{align}
For $\xi=0$ and $B(t_{i-1,\ell}, \what{X}^{J,L}(t_{i-1,\ell}))= B(t_{i-1,\ell})$ the equation~\eqref{def:disc conv} is a discrete analogue of the stochastic convolution. 
%
The Fourier coefficients of~\eqref{def:disc conv} are given by
\begin{align*}
&\dsconvjk{ B(\cdot, \what{X}^{J,L}(\cdot)) }(\tau_{\eta}) 
:=
\innprod{\dstoconv{ B(\cdot, \what{X}^{J,L}(\cdot)) }(\tau_{\eta})}
{h_{{j}}} \\
&= 
\sumltrunc \sum_{\tau_1\le t_{i,\ell} \le \tau_{\eta}}
\sqrt{ q_{\ell}}\mathfrak{R}_{j}(t_{i-1,\ell},\tau_{\eta})
\innprod{B(t_{i-1,\ell}, \what{X}^{J,L}(t_{i-1,\ell})) h_{{\ell}}}{h_{{j}}}
( \beta_{{\ell}}(t_{i,\ell}) - \beta_{{\ell}}(t_{i-1,\ell})),
\end{align*}
for $j\in\{1,\dotsc,J\}$ and $\eta\in\{1,\dotsc,N\}$.
Then, {noting that by the assumptions on $A$ we have 
	$((I-\lambda A)^{-1})^*=(I-\lambda A)^{-1}$ for $\lambda\in (0,\infty)$, the Fourier coefficients of the discretised solution are given by}
\begin{align*}
\what{X}^{J,L}_{{j}}(\tau_{\eta})
=
\mathfrak{R}_{j}(\tau_0,\tau_{\eta})\innprod{\xi}{h_{{j}}}
+
\dsconvjk{ B(\cdot, \what{X}^{J,L}(\cdot)) }(\tau_{\eta}).
\end{align*}

For any $r\ge0$ we have
\begin{align}
\mathbb{E}\|
\what{X}^{J,L} (\tau_{\eta})
\|_{ D(A^{r})}^{2}
=&
\sumjtrunc
\lambda_j^{2 r}
\big|
\mathfrak{R}_{j}(\tau_0,\tau_{\eta})\innprod{\xi}{h_{{j}}}
\big|^2
+
\sumjtrunc
\lambda_j^{2 r}
\mathbb{E}
\big|
\dsconvjk{ B(\cdot, \what{X}^{J,L}(\cdot)) }(\tau_{\eta})
\big|^2.
\label{eq:decomp norm}
\end{align}
\red{Our first goal is to estimate the second term in the right hand side of \eqref{eq:decomp norm}.}
\red{We see this term as the stochastic integral of a representation of an elementary process. }
%
%
%

Let $
\mathscr{P}_{{\ell}}x := \innprod{x}{h_{{\ell}}} h_{{\ell}}$ for $\ell\ge1$, 
and let $\iota\ge0$ be the index from Assumption~\ref{assump:cond on B}. 
For $\nu\in \{1,\dotsc,\eta\}$, we define 
an $\HS(H_0, D(A^{\iota}))$-valued random variable $(\phi^{J,(\eta)}_{{\ell}})_{\nu-1}$ by 
\begin{subequations}
	\begin{empheq}[left={
		\hspace{-0.2em}
		(\phi^{J,(\eta)}_{{\ell}})_{\nu-1}
		\!:=\!
		\empheqlbrace}]{alignat=2}
	&\mathscr{P}_J {R}(s_{\nu,\ell},\tau_{\eta};A)
	B (
	s_{\nu,\ell},\what{X}^{J,L}(s_{\nu,\ell})
	)
	\mathscr{P}_{{\ell}}
	& &\quad\text{ if }\ell\in\Xi_{\nu}
	\label{eq:ell st t_i,ell = tau} \\[10pt]
	&0_{H_0\to H}
	& &\quad\text{ if }\ell\not\in\Xi_{\nu},
	\label{eq:ell st t_i,ell exceeds tau}
	\end{empheq}
\end{subequations}
where 
\begin{align}
\Xi_{\nu}:=
\big\{
\ell\in\{1,\dotsc,L\}\bigmid
\ell\in\calK_{\mu}\text{ for some }\mu\in\{\nu,\dotsc,\eta\}
\big\}.\label{eq:def Xi_nu}
\end{align}
%

We elaborate on the notation. First, note the following: for $\ell\not\in\calK_{\nu}$, $\nu\in\{0,\dotsc,\eta\}$ if the index $i'\in\{1,\dotsc,n_{\ell}\}$ is such that $s_{\nu,\ell}=t_{i'-1,\ell}$, then we have $\tau_{\nu}<t_{i',\ell}$.
The separate treatment~\eqref{eq:ell st t_i,ell exceeds tau} corresponds to the construction of the algorithm: 
suppose $\ell\in\{1,\dotsc,L\}$ and $i^{*}\in\{1,\dotsc,n_{\ell}\}$  satisfy $s_{\eta,\ell}=t_{i^{*}-1,\ell}$ and $\tau_{\eta}<t_{i^{*},\ell}$, then the evaluation $\beta_{{\ell}}(t_{i^{*},\ell})$ of the Brownian motion $\beta_{{\ell}}$ at $t_{i^{*},\ell}$ is not used to obtain $\what{X}^{J,L}_{{j}}(\tau_{\eta})${;} only up to $\beta_{{\ell}}(t_{0,\ell}),\dots,\beta_{{\ell}}(t_{i^{*}-1,\ell})$ are used.

Let us define the elementary process 
$\Phi^{J,(\eta)}_{{\ell}}\from\Omega\times [0,\tau_{\eta}]\to {\HS(H_0, D(A^{\iota}))}$ by 
\begin{align}
\Phi^{J,(\eta)}_{{\ell}}(\omega,t)
:=
\sum_{\nu=1}^{\eta}
(\phi^{J,(\eta)}_{{\ell}})_{\nu-1}(\omega)
\,\indic_{(\tau_{\nu-1},\tau_{\nu}]}(t).
\label{eq:def-Phi}
\end{align}
Then, we have the following.
\begin{lemma}\label{lem:sto int indentity}
	Let $\dsconvjk{ B(\cdot, \what{X}^{J,L}(\cdot)) }(\cdot)$ be defined by~\eqref{def:disc conv} 
	and let Assumption~\ref{assump:cond on B} hold with $\iota\ge0$. Then, for $j=1\dotsc,J$ 
	we have 
	\begin{align*}
	\dsconvjk{ B(\cdot, \what{X}^{J,L}(\cdot)) }(\tau_{\eta})
	=
	\innprod{
		\int_{0}^{\tau_\eta}
		\sumltrunc \Phi^{J,(\eta)}_{{\ell}}(s) \dW(s)
	}{h_{{j}}}.
	\end{align*}
\end{lemma}
\begin{proof}
	Fix $\eta\in\{1,\dots, N\}$. Let $\calS_{\mu}:=\calK_{\eta-\mu}\setminus(\bigcup_{\mu'\in\{0,\dotsc,\mu-1\}}\calK_{\eta-\mu'})$ for $\mu,\eta\in\{1,\dots, N\}$ with $\mu\le\eta$, and let 
	$\calS_{0}:=\calK_{\eta}$.
	Then, we have
	\begin{align*}
	\begin{split}
	&\dsconvjk{ B(\cdot, \what{X}^{J,L}(\cdot)) }(\tau_{\eta}) \\
	&=
	\sum_{\nu=1}^{\eta}
	\sum_{ \ell\in \calK_{\eta}}
	\sqrt{ q_{\ell}}\mathfrak{R}_{j}( s_{\nu,\ell} ,\tau_{\eta})
	\innprod{B(  s_{\nu,\ell} , \what{X}^{J,L}( s_{\nu,\ell} )) h_{{\ell}}}{h_{{j}}}
	( \beta_{{\ell}}( \tau_{\nu} ) - \beta_{{\ell}}( \tau_{\nu-1} ))  \\
	&+
	\sum_{\nu=1}^{\eta-1}
	\sum_{ \ell\in \calS_{1}}
	\sqrt{ q_{\ell}}\mathfrak{R}_{j}( s_{\nu,\ell} ,\tau_{\eta})
	\innprod{B(  s_{\nu,\ell} , \what{X}^{J,L}( s_{\nu,\ell} )) h_{{\ell}}}{h_{{j}}}
	( \beta_{{\ell}}( \tau_{\nu} ) - \beta_{{\ell}}( \tau_{\nu-1} ))  \\
	&\shortvdotswithin{+}
	&+
	\sum_{\nu=1}^{\eta-\mu}
	\sum_{ \ell\in \calS_{\mu}}
	\sqrt{ q_{\ell}}\mathfrak{R}_{j}( s_{\nu,\ell} ,\tau_{\eta})
	\innprod{B(  s_{\nu,\ell} , \what{X}^{J,L}( s_{\nu,\ell} )) h_{{\ell}}}{h_{{j}}}
	( \beta_{{\ell}}( \tau_{\nu} ) - \beta_{{\ell}}( \tau_{\nu-1} ))  \\
	&\shortvdotswithin{+}
	&+
	\sum_{ \ell\in \calS_{\eta-1}}
	\sqrt{ q_{\ell}}\mathfrak{R}_{j}( s_{1,\ell} ,\tau_{\eta})
	\innprod{B(  s_{1,\ell} , \what{X}^{J,L}( s_{1,\ell} )) h_{{\ell}}}{h_{{j}}}
	( \beta_{{\ell}}( \tau_{1} ) - \beta_{{\ell}}( \tau_{0} )).
	\end{split}
	\end{align*}
	Further, we can rewrite the above as
	\begin{align*}
	&\dsconvjk{ B(\cdot, \what{X}^{J,L}(\cdot)) }(\tau_{\eta}) \\
	=&\!\!\sum_{\mu=0}^{\eta-1}\!\!\,
	\sum_{\nu=1}^{\eta-\mu}\!\!\,
	\sum_{ \ell\in \calS_{\mu}}\!\,
	\big\langle\!B(  s_{\nu,\ell} , \what{X}^{J,L}( s_{\nu,\ell} )) \mathscr{P}_{{\ell}}
	( W( \tau_{\nu} ) - W( \tau_{\nu-1} ))
	, R( s_{\nu,\ell} ,\tau_{\eta};A)\mathscr{P}_J h_{{j}}\!\big\rangle.
	\end{align*}
	By the assumptions on $A$ we have 
	$((I-\lambda A)^{-1})^*=(I-\lambda A)^{-1}$ for $\lambda\in (0,\infty)$, and thus
	\[
	\dsconvjk{ B(\cdot, \what{X}^{J,L}(\cdot)) }(\tau_{\eta})
	=
	\innprod{\sum_{\nu=1}^{\eta}\!
		\bigg(
		\sumltrunc
		(\phi^{J,(\eta)}_{{\ell}})_{\nu-1}
		\bigg)
		\big(
		W(\tau_{\nu})-W(\tau_{\nu-1})
		\big)
	}{h_{{j}}\!}.
	\]
	By definition of the stochastic integral of elementary processes the statement follows.
\end{proof}
%
%
%
%
Using the previous result, we {obtain the following estimate.}
\begin{proposition}\label{prop:Burkholder}
	Let Assumption~\ref{assump:cond on B} hold.
	Let $\eta\in\{1,\dotsc,N \}$. 
	For $p\ge1$, suppose that the process defined by~\eqref{eq:ell st t_i,ell = tau}--\eqref{eq:ell st t_i,ell exceeds tau} satisfies 
	\begin{align}
	\mathbb{E}\bigg[{
		\sum_{\nu=1}^{\eta}
		\bigg\|
		\sumltrunc (\phi^{J,(\eta)}_{{\ell}})_{\nu-1}
		\bigg\|_{ \HS(H_0,  D(A^{\iota}) ) }^{2}
		(\tau_{\nu} - \tau_{\nu-1})}\bigg]<\infty.\label{eq:assump Burkholder}
	\end{align}
	Then, we have
	\begin{align}
	\mathbb{E}\bigg[
	{\big\|}&
	\dstoconv{ B(\cdot, \what{X}^{J,L}(\cdot)) }(\tau_{{\eta}})
	\big\|_{ D(A^{\iota})}^{{2}}
	\bigg]
	\le
	\mathbb{E}
	{\bigg[
		\sum_{\nu=1}^{\eta}
		\bigg\|
		\sumltrunc (\phi^{J,(\eta)}_{{\ell}})_{\nu-1}
		\bigg\|_{ \HS(H_0,  D(A^{\iota}) ) }^{2}
		(\tau_{\nu} - \tau_{\nu-1})
	}
	\bigg].\label{eq:E[|(tau_eta)|Lp]}
	\end{align}
\end{proposition}
%
%
\begin{proof}
	For any $\eta\in\{1,\dotsc,N \}$, from Lemma~\ref{lem:sto int indentity} we have
	\begin{align*}
	\mathbb{E}\big[&
	\big\|
	\dstoconv{ B(\cdot, \what{X}^{J,L}(\cdot)) }(\tau_{\eta})
	\big\|_{ D(A^{\iota})}^{{2}}
	\big]=
	\mathbb{E}\bigg[
	\sumjtrunc
	\lambda_{j}^{2r}
	\bigg|
	\bigg\langle{\int_{0}^{\tau_\eta}
		\sumltrunc \Phi^{J,(\eta)}_{{\ell}}(s) \dW(s)},{h_{{j}}}
	\bigg\rangle
	\bigg|^2
	\bigg].
	\end{align*}
	It follows that
	\begin{align*}
	\mathbb{E}\bigg[
	\sumjtrunc
	\bigg|
	\bigg\langle{\int_{0}^{\tau_\eta}
		\sumltrunc \Phi^{J,(\eta)}_{{\ell}}(s) \dW(s)}
	&,
	{\lambda_{j}^{r} h_{{j}}}\bigg\rangle
	\bigg|^2
	\bigg]
	\le
	\mathbb{E}
	\bigg[
	\bigg\|
	\int_{0}^{\tau_\eta}
	\sumltrunc \Phi^{J,(\eta)}_{{\ell}}(s) \dW(s)
	\bigg\|_{  D(A^{\iota}) }^{{2}}
	\bigg]\\
	&{=
		\mathbb{E}
		\bigg[
		\int_{0}^{\tau_\eta}
		\bigg\|
		\sumltrunc \Phi^{J,(\eta)}_{{\ell}}(s)
		\bigg\|_{ \HS(H_0,  D(A^{\iota}) ) }^{2}
		\ds
		\bigg]}
	\\
	&{= 
		\mathbb{E}
		\bigg[
		\sum_{\nu=1}^{\eta}
		\bigg\|
		\sumltrunc (\phi^{J,(\eta)}_{{\ell}})_{\nu-1}
		\bigg\|_{ \HS(H_0,  D(A^{\iota}) ) }^{2}
		(\tau_{\nu} - \tau_{\nu-1})
		\bigg]<\infty,}
	\end{align*}
	{where in the first equality It\^{o}'s isometry, and in the last inequality the  condition~\eqref{eq:assump Burkholder} is used}. Thus, the statement follows.
\end{proof}

We need the following estimate for the process $(\phi^{J,(\eta)}_{{\ell}})_{\nu-1}$ as in~\eqref{eq:ell st t_i,ell = tau} and 
\eqref{eq:ell st t_i,ell exceeds tau} in terms of the Hilbert--Schmidt norm.
\begin{lemma}\label{lem:bd phi indep eta}
	Suppose that Assumption~\ref{assump:cond on B} is satisfied. 
	Fix an arbitrary integer $\eta\in\{1,\dotsc,N\}$. 
	Then, for any $\nu\in\{0,\dotsc,\eta \}$, we have
	\begin{align*}
	&\bigg\|
	\sumltrunc (\phi^{J,(\eta)}_{{\ell}})_{\nu-1}
	\bigg\|_{ \HS(H_0,  D(A^{\iota}) ) }
	=
	\bigg\|
	{\sum_{\ell\in\Xi_{\nu}}} 
	(\phi^{J,(\eta)}_{{\ell}})_{\nu-1}
	\bigg\|_{ \HS(H_0,  D(A^{\iota}) ) }
	\\
	&\le
	\bigg(
	{\sum_{\ell\in\Xi_{\nu}}} 
	\sumjtrunc
	\lambda_j^{2\iota}|\mathfrak{R}_j(s_{\nu,\ell},\tau_{\eta})|^2
	\Big|
	\innprod{
		B(
		s_{\nu,\ell},\what{X}^{J,L}(s_{\nu,\ell})
		)
		\sqrt{ q_\ell}h_{{\ell}}
	}{
		h_{{j}}
	}
	\Big|^2
	\bigg)^{\frac12},
	\end{align*}
	where $\Xi_{\nu}$ is defined by~\eqref{eq:def Xi_nu}.
\end{lemma}
\begin{proof}
	Note that if
	$\ell\not\in\Xi_{\nu}$, then
	$\norm{
		(\phi^{J,(\eta)}_{{\ell}})_{\nu-1}
		\sqrt{ q_\ell}h_{{\ell}}
	}_{  D(A^{\iota}) }^2=0$. 
	Thus, noting that $\mathscr{P}_{{\ell}}h_{{\ell'}}=0$ 
	unless {$\ell=\ell'$}, from the definition of $(\phi^{J,(\eta)}_{{\ell}})_{\nu-1}$ we have
	\begin{align*}
	\bigg\|
	\sumltrunc (\phi^{J,(\eta)}_{{\ell}})_{\nu-1}
	\bigg\|_{ \HS(H_0,  D(A^{\iota}) ) }^2
	&=
	{\sum_{\ell'=1}^{L}} 
	\bigg\|
	(\phi^{J,(\eta)}_{{\ell'}})_{\nu-1}
	\sqrt{ q_{\ell'}} h_{{\ell'}}
	\bigg\|_{  D(A^{\iota})  }^2 \\
	&=
	{\sum_{\ell'\in\Xi_{\nu}}} 
	\bigg\|
	(\phi^{J,(\eta)}_{{\ell'}})_{\nu-1}
	\sqrt{ q_{\ell'}} h_{{\ell'}}
	\bigg\|_{  D(A^{\iota})  }^2.
	\end{align*}
	Fix $\ell\in\Xi_{\nu}$. For any $\eta\in\{1,\dotsc,N\}$ and $\nu\in\{1,\dotsc,\eta \}$ we have 
	\begin{align*}
	&\norm{
		(\phi^{J,(\eta)}_{{\ell}})_{\nu-1}
		\sqrt{ q_\ell}h_{{\ell}}
	}_{  D(A^{\iota}) }^2 \\
	&
	=
	\sumjall
	\lambda_j^{2\iota}
	\Big|
	\innprod{
		\mathscr{P}_J
		{R}(s_{\nu,\ell},\tau_{\eta};A)
		B(
		s_{\nu,\ell},\what{X}^{J,L}(s_{\nu,\ell})
		)
		\sqrt{ q_\ell}h_{{\ell}}
	}{
		h_{{j}}
	}
	\Big|^2
	\\
	&\le
	\sumjtrunc
	\lambda_j^{2\iota}|\mathfrak{R}_j(s_{\nu,\ell},\tau_{\eta})|^2
	\Big|
	\innprod{
		B(
		s_{\nu,\ell},\what{X}^{J,L}(s_{\nu,\ell})
		)
		\sqrt{ q_\ell}h_{{\ell}}
	}{
		h_{{j}}
	}
	\Big|^2.
	\end{align*}
	Hence, the statement follows.
\end{proof}
%
%
The following lemma is important to show the maximal regularity estimate of the same form as the continuous counterpart~\eqref{eq:maximal reg shown}, studied in \cite[Proposition 6.18]{DaPrato.D_2014_book} and \cite{DaPrato.G_1982_reg_convo}.
\begin{lemma}\label{lem:reslent get 1/lambda}
	For any $j\ge1$, $\ell\ge1$, and $i\in\{1,\dotsc,n_\ell\}$, we have
	\begin{align*}
	\sum_{t_{i,\ell}\le\tau_{\eta}\le\tau_{N}}
	|\mathfrak{R}_j(t_{i-1,\ell},\tau_{\eta})|^2
	(\tau_{\eta} - \tau_{\eta-1})\le
	\frac{2}
	{\lambda_j},
	\end{align*}
	where $\mathfrak{R}_j(\cdot,\cdot)$ is defined by~\eqref{def:Gamma}.
\end{lemma}
\begin{proof}
	For $\tau_{\eta_0}\in \{\tau_{0},\dotsc,\tau_{N} \}$ define 
	a continuous interpolation $\mathcal{S}_j(\tau_{\eta_0},\cdot)\from[0,\TMAX]\to\bbR$ of $\mathfrak{R}_j(\tau_{\eta_0},\tau_{\eta})$ by 
	\begin{align}
	\mathcal{S}_j(\tau_{\eta_0},t)
	:=
	\prod_{\nu={\eta_0}+1}^{N} \frac{1}
	{1 + \lambda_{j}(t\wedge \tau_\nu - t\wedge \tau_{\nu-1})},\qquad t\in[0,\TMAX].
	\label{eq:S_j}
	\end{align}
	Then, for $t\in(\tau_{\eta-1},\tau_{\eta}]$, $\eta\in\{1,\dotsc,N\}$, we have
	\begin{align*}
	\mathcal{S}_j(\tau_{\eta_0},t)\indic_{ \{(\tau_{\eta-1},\tau_{\eta}] \} }(t)&=\mathcal{S}_j(\tau_{\eta_0},t)
	\ge 
	\mathfrak{R}_j(\tau_{\eta_0},\tau_{\eta})\indic_{ \{(\tau_{\eta-1},\tau_{\eta}] \} }(t).
	\end{align*}
	Further, for $\ell=1,\dotsc,L$ and $i=1,\dotsc,n_{\ell}$, let 
	\[
	\tau_{\eta^*}:=\tau_{\eta^*(i,\ell)}:=t_{i,\ell}.
	\]
	Then, we have
	\begin{align*}
	\sum_{t_{i,\ell}\le\tau_{\eta}\le\tau_{N}}&
	|\mathfrak{R}_j(t_{i-1,\ell},\tau_{\eta})|^2
	(\tau_{\eta} - \tau_{\eta-1})
	\\
	&
	\le
	\int_{\tau_{\eta^*-1}}^{\TMAX}
	\sum_{\eta=\eta^*(i,\ell)}^{N}
	|\mathcal{S}_j(t_{i-1,\ell},s)|^2
	\indic_{ \{ (\tau_{\eta-1},\tau_{\eta}] \} }(s)
	\ds \\
	&=\int_{\tau_{\eta^*-1}}^{\TMAX}
	|\mathcal{S}_j(t_{i-1,\ell},s)|^2
	\ds 
	\le
	\int_{t_{i-1,\ell}}^{\TMAX}
	|\mathcal{S}_j(t_{i-1,\ell},s)|^2
	\ds.
	\end{align*}
	For $t\in[t_{\kappa-1,\ell},t_{\kappa-1,\ell}]$ with $\kappa\ge i$,
	the elementary inequality 
	$
	\frac{1}{1+(b-a)}\frac{1}{1+(c-b)}\le\frac{1}{1+(c-a)}
	$ $(0\le a\le b\le c)$ implies
	\begin{align*}
	&\mathcal{S}_j(t_{i-1,\ell},t)
	\le 
	\frac{1}{(1+\lambda_j \frac1{n_{\ell}})^{\kappa-i}}\cdot
	\frac{1}{1+\lambda_j(t-t_{\kappa-1,\ell})},
	\end{align*}
	and therefore
	\begin{align*}
	&\int_{t_{i-1,\ell}}^{\TMAX}
	|\mathcal{S}_j(t_{i-1,\ell},s)|^2
	\ds
	=
	\sum_{\kappa=i}^{n_{\ell}}
	\int_{t_{\kappa-1,\ell}}^{t_{\kappa,\ell}}
	|\mathcal{S}_j(t_{i-1,\ell},s)|^2
	\ds \\
	&\le
	\sum_{\kappa=i}^{n_{\ell}}
	\frac{1}{(1+\frac{\lambda_j}{n_{\ell}})^{2\kappa-2i}}
	\int_{t_{\kappa-1,\ell}}^{t_{\kappa,\ell}}
	\frac{1}{(1+\lambda_j(s-t_{\kappa-1,\ell}))^2}
	\ds \\
	&=\sum_{\kappa=i}^{n_{\ell}}
	\frac{1}{(1+\frac{\lambda_j}{n_{\ell}})^{2\kappa-2i}}
	\frac{1}{\lambda_j + 1/(t_{\kappa,\ell}-t_{\kappa-1,\ell})}
	\le
	\frac{1}{\lambda_j + {n_{\ell}}}
	\sum_{\kappa=i}^{n_{\ell}}
	\frac{1}{(1+\frac{\lambda_j}{n_{\ell}})^{2\kappa-2i}}.
	\end{align*}
	If $\frac{\lambda_j}{n_{\ell}}\ge1$, then 
	$
	\frac{1}{\lambda_j + {n_{\ell}}}
	\sum_{\kappa=i}^{n_{\ell}}
	\frac{1}{(1+{\lambda_j}/{n_{\ell}})^{2\kappa-2i}}
	\le \frac{2}{\lambda_j}
	$, 
	and otherwise $(1+\frac{\lambda_j}{n_{\ell}})^2\le4$ and thus
	\begin{align*}
	\frac{1}{\lambda_j + {n_{\ell}}}
	\sum_{\kappa=i}^{n_{\ell}}
	\frac{1}{(1+\frac{\lambda_j}{n_{\ell}})^{2\kappa-2i}}
	&\le
	\frac{1}{{n_{\ell}}}
	\frac1{1-1/{(1+\frac{\lambda_j}{n_{\ell}})^2}}\le
	\frac{4}
	{2\lambda_j
		+\lambda_j^2/n_{\ell}
	}
	\le
	\frac{2}
	{\lambda_j}.
	\end{align*}
	Hence, we have
	$
	\sum_{t_{i,\ell}\le\tau_{\eta}\le\tau_{N}}
	|\mathfrak{R}_j(t_{i-1,\ell},\tau_{\eta})|^2
	(\tau_{\eta} - \tau_{\eta-1})\le
	\frac{2}
	{\lambda_j}$, as claimed.
\end{proof}
%
%
We are ready to state our main result.
\begin{theorem}\label{thm:sto conv with X lp reg}
	Suppose Assumption~\ref{assump:cond on B} is satisfied with some $\red{\iota\in[0,1/2]}$. Then, we have
	\begin{align*}
	\sum_{\eta=1}^{N}&
	\mathbb{E}\Big[
	\norm{
		\dstoconv{ B(\cdot, \what{X}^{J,L}(\cdot)) }(\tau_{\eta})
	}_{  D(A^{^{\iota+1/2}}) }^{2}
	\Big]
	(\tau_{\eta} - \tau_{\eta-1}) \\
	&\le
	{2}
	\mathbb{E}\bigg[
	\sumlall
	\sum_{i=1}^{n_{\ell}}
	\norm{
		\mathscr{P}_J
		B(
		t_{i-1,\ell},\what{X}^{J,L}(t_{i-1,\ell})
		)
		\mathscr{P}_L
		\sqrt{ q_\ell}h_{{\ell}}
	}_{ D(A^{\iota})}^2
	(t_{i,\ell} - t_{i-1,\ell})
	\bigg].
	\end{align*}
	{In particular, $\overline{X}^{J,L}$ defined as in \eqref{eq::ch-maximal-unif} satisfies}
	\begin{align*}
	\sum_{i=1}^{N}
	\mathbb{E}\Big[&
	\norm{
		\dstoconv{ B(\cdot, {\overline{X}^{J,L}}(\cdot)) }(t_{i})
	}_{  D(A^{^{\iota+1/2}}) }^{2}
	\Big]
	(t_{i} - t_{i-1}) \\
	&\le
	{2}
	\sum_{i=1}^{N}
	\mathbb{E}\bigg[
	\norm{
		\mathscr{P}_J
		B(
		t_{i-1},{\overline{X}^{J,L}}(t_{i-1})
		)
		\mathscr{P}_L
	}_{\HS(H_0, D(A^{\iota}))}^2
	\bigg]
	(t_{i} - t_{i-1}).
	\end{align*}
\end{theorem}
\begin{proof}
	We first show that for $\eta=1,\dots,N$, we have
	\begin{align}
	\mathbb{E}\bigg[
	\sum_{\nu=1}^{\eta}
	\bigg\|
	\sumltrunc (\phi^{J,(\eta)}_{{\ell}})_{\nu-1}
	\bigg\|_{ \HS(H_0,  D(A^{^{\iota+1/2}}) ) }^{2}
	(\tau_{\nu} - \tau_{\nu-1})
	\bigg]<\infty.\label{eq:intm bd}
	\end{align}
	In view of Lemma~\ref{lem:bd phi indep eta}, we have
	\begin{align}
	&
	\sum_{\eta=1}^{N}
	\mathbb{E}\bigg[
	\sum_{\nu=1}^{\eta}
	\bigg\|
	\sumltrunc (\phi^{J,(\eta)}_{{\ell}})_{\nu-1}
	\bigg\|_{ \HS(H_0,  D(A^{^{\iota+1/2}}) ) }^{2}
	(\tau_{\nu} - \tau_{\nu-1})
	\bigg]
	(\tau_{\eta} - \tau_{\eta-1})\notag\\
	&\le
	\sum_{\eta=1}^{N}
	\mathbb{E}\bigg[
	\sumjtrunc 
	\sum_{\nu=1}^{\eta} \sum_{\ell\in\Xi_{\nu}} 
	\lambda_j^{2\iota+1}|\mathfrak{R}_j(s_{\nu,\ell},\tau_{\eta})|^2
	\notag\\[-1.1pt]
	&\qquad\qquad\times\Big|
	\innprod{
		B(
		s_{\nu,\ell},\what{X}^{J,L}(s_{\nu,\ell})
		)
		\sqrt{ q_\ell}h_{{\ell}}
	}{
		h_{{j}}
	}
	\Big|^2
	(\tau_{\nu} - \tau_{\nu-1})
	\bigg]
	(\tau_{\eta} - \tau_{\eta-1}) \notag\\[-1.1pt]
	&= 
	\mathbb{E}\bigg[
	\sumjtrunc 
	\sumltrunc
	\sum_{\eta=1}^{N}
	\sum_{\tau_{1}\le t_{i,\ell}\le \tau_{\eta}}
	\lambda_j^{2\iota+1}|\mathfrak{R}_j(t_{i-1,\ell},\tau_{\eta})|^2
	\notag \\[-1.1pt]
	&\qquad
	\times\Big|
	\innprod{
		B(
		t_{i-1,\ell},\what{X}^{J,L}(t_{i-1,\ell})
		)
		\sqrt{ q_\ell}h_{{\ell}}
	}{
		h_{{j}}
	}
	\Big|^2
	(t_{i,\ell}-t_{i-1,\ell})
	(\tau_{\eta} - \tau_{\eta-1})
	\bigg].
	\label{eq:t_il sum}
	\end{align}
	Since 
	$
	\bigcup_{\eta=1}^N \bigcup_{\tau_{1}\le t_{i,\ell}\le \tau_{\eta}}
	\{ \tau_{\eta},t_{i,\ell}\}
	=\bigcup_{i=1}^{n_{\ell}} \bigcup_{t_{i,\ell}\le \tau_{\eta}\le \tau_{N}}
	\{ \tau_{\eta},t_{i,\ell}\}
	$, 
	the right hand side of~\eqref{eq:t_il sum} can be rewritten as
	\begin{align}
	&\mathbb{E}\bigg[
	\sumjtrunc 
	\sumltrunc
	\sum_{i=1}^{n_{\ell}}
	\sum_{t_{i,\ell}\le\tau_{\eta}\le\tau_{N}}
	\lambda_j^{2\iota+1}|\mathfrak{R}_j(t_{i-1,\ell},\tau_{\eta})|^2
	\notag \\
	&\qquad
	\times\Big|
	\innprod{
		B(
		t_{i-1,\ell},\what{X}^{J,L}(t_{i-1,\ell})
		)
		\sqrt{ q_\ell}h_{{\ell}}
	}{
		h_{{j}}
	}
	\Big|^2
	(t_{i,\ell} - t_{i-1,\ell})
	(\tau_{\eta} - \tau_{\eta-1})
	\bigg]\notag \\
	&=\mathbb{E}\bigg[
	\sumjtrunc 
	\sumltrunc
	\sum_{i=1}^{n_{\ell}}
	\lambda_j^{2\iota+1}
	\Big|
	\innprod{
		B(
		t_{i-1,\ell},\what{X}^{J,L}(t_{i-1,\ell})
		)
		\sqrt{ q_\ell}h_{{\ell}}
	}{
		h_{{j}}
	}
	\Big|^2\notag \\
	&\quad \times(t_{i,\ell} - t_{i-1,\ell})
	\sum_{t_{i,\ell}\le\tau_{\eta}\le\tau_{N}}
	|\mathfrak{R}_j(t_{i-1,\ell},\tau_{\eta})|^2
	(\tau_{\eta} - \tau_{\eta-1})
	\bigg].\label{eq:earn lambda from resolvent}
	\end{align}
	From Lemma~\ref{lem:reslent get 1/lambda}, \eqref{eq:t_il sum} and \eqref{eq:earn lambda from resolvent}, due to Assumption~\ref{assump:cond on B} we have~\eqref{eq:intm bd}.
	
	From~\eqref{eq:intm bd}, we note that Proposition~\ref{prop:Burkholder} implies
	\begin{align*}
	\sum_{\eta=1}^{N}
	\mathbb{E}\Big[&
	\norm{ 
		\dstoconv{ B(\cdot, \what{X}^{J,L}(\cdot)) }(\tau_{\eta})
	}_{  D(A^{^{\iota+1/2}}) }^{2}
	\Big]
	(\tau_{\eta} - \tau_{\eta-1})\nonumber \\
	&\le
	{\sum_{\eta=1}^{N}}
	\mathbb{E}\bigg[
	\sum_{\nu=1}^{\eta}
	\bigg\|
	\sumltrunc (\phi^{J,(\eta)}_{{\ell}})_{\nu-1}
	\bigg\|_{ \HS(H_0,  D(A^{^{\iota+1/2}}) ) }^{2}
	(\tau_{\nu} - \tau_{\nu-1})
	\bigg]
	(\tau_{\eta} - \tau_{\eta-1}).
	\end{align*}
	Therefore, again from Lemma~\ref{lem:reslent get 1/lambda} together with~\eqref{eq:t_il sum} and~\eqref{eq:earn lambda from resolvent} 
	we obtain 
	\begin{align*}
	\sum_{\eta=1}^{N}
	\mathbb{E}\Big[&
	\norm{ 
		\dstoconv{ B(\cdot, \what{X}^{J,L}(\cdot)) }(\tau_{\eta})
	}_{  D(A^{^{\iota+1/2}}) }^{2}
	\Big]
	(\tau_{\eta} - \tau_{\eta-1})\\
	&\le
	{2}
	\mathbb{E}\bigg[
	\sumjtrunc 
	\sumltrunc
	\sum_{i=1}^{n_{\ell}}
	\lambda_j^{2\iota}
	\Big|
	\innprod{
		B(
		t_{i-1,\ell},\what{X}^{J,L}(t_{i-1,\ell})
		)
		\sqrt{ q_\ell}h_{{\ell}}
	}{
		h_{{j}}
	}
	\Big|^2
	(t_{i,\ell} - t_{i-1,\ell})
	\bigg] \\
	&=
	{2}
	\mathbb{E}\bigg[
	\sumlall
	\sum_{i=1}^{n_{\ell}}
	\norm{
		\mathscr{P}_J
		B(
		t_{i-1,\ell},\what{X}^{J,L}(t_{i-1,\ell})
		)
		\mathscr{P}_L
		\sqrt{ q_\ell}h_{{\ell}}
	}_{ D(A^{\iota})}^2
	(t_{i,\ell} - t_{i-1,\ell})
	\bigg].
	\end{align*}
	When $n_{\ell}=N$ for all $\ell\in\{1,\dotsc, L\}$, we have 
	$t_{i,\ell} - t_{i-1,\ell}=t_{i} - t_{i-1}$ ($i=1,\dotsc,N$). Thus, repeating the same argument as above completes the proof.
\end{proof}
As a consequence of the previous result, given a suitable regularity of the initial condition, the approximate solution has the spatial regularity ``one-half smoother''---the same as the continuous counterpart \cite{DaPrato.D_2014_book}---than the range of the operator $B(t,x)$. 
\begin{corollary}\label{cor:max estim}
	Suppose Assumption~\ref{assump:cond on B} is satisfied with some $\red{\iota\in[0,1/2]}$, and let $\xi\in D(A^{\iota})$. Then, we have
	\begin{align*}
	\bigg(&
	\sum_{\eta=1}^{N}
	\mathbb{E}\big[	\|
	\what{X}^{J,L} (\tau_{\eta})
	\|_{ D(A^{^{\iota+1/2}})}^2
	\big]
	(\tau_{\eta} - \tau_{\eta-1})
	\bigg)^{\frac12}\preceq
	\big\|
	\mathscr{P}_{J} \xi 
	\big\|_{  D(A^{\iota}) } \\
	& +
	\bigg(
	\mathbb{E}\bigg[
	\sumlall
	\sum_{i=1}^{n_{\ell}}
	\norm{
		\mathscr{P}_J
		B(
		t_{i-1,\ell},\what{X}^{J,L}(t_{i-1,\ell})
		)
		\mathscr{P}_L
		\sqrt{ q_\ell}h_{{\ell}}
	}_{ D(A^{\iota})}^2
	(t_{i,\ell} - t_{i-1,\ell})
	\bigg]
	\bigg)^{\frac12}.
	\end{align*}
	{In particular, $\overline{X}^{J,L}$ defined as in \eqref{eq::ch-maximal-unif} satisfies}
	\begin{align*}
	&{\bigg(
		\sum_{i=1}^{N}
		\bbE\big[\!	\|
		\overline{X}^{J,L} (t_{i})
		\|_{D(A^{\iota+1/2})}^2
		\big]
		(t_{i} - t_{i-1})
		\bigg)^{\frac12}}\\
	&
	{
		\preceq
		\big\|
		\mathscr{P}_J \xi
		\big\|_{ D(A^{\iota}) } }
	{+
		\bigg(
		\sum_{i=1}^{N}
		\bbE\bigg[
		\norm{
			\mathscr{P}_J
			B(t_{i-1},
			\overline{X}^{J,L}(t_{i-1})
			)
			\mathscr{P}_L 
		}_{\HS(H_0,D(A^{\iota}))}^2
		\bigg]
		(t_{i} - t_{i-1})\bigg)^{\frac12}.}
	\end{align*}
\end{corollary}
\begin{proof}
	From Lemma~\ref{lem:reslent get 1/lambda} we have
	\begin{align*}
	\sum_{\eta=1}^{N}&
	\mathbb{E}\big[
	\norm{{R}(\tau_0,\tau_\eta;A)\mathscr{P}_{J} \xi  }_{  D(A^{^{\iota+1/2}}) }^2
	\big]
	(\tau_{\eta} - \tau_{\eta-1})\nonumber\\
	&
	=
	\mathbb{E}\big[
	\sumjtrunc
	\lambda_j^{2\iota+1}
	\big|\innprod{\xi}{h_{{j}}}\big|^2
	\sum_{\eta=1}^{N}
	\big|\mathfrak{R}_{j}(\tau_0,\tau_{\eta})\big|^2
	(\tau_{\eta} - \tau_{\eta-1})
	\big]
	\\
	&\le2
	\mathbb{E}\big[
	\sumjtrunc
	\lambda_j^{2\iota}
	\big|\innprod{\xi}{h_{{j}}}\big|^2
	\big].
	\end{align*}
	Then, from~\eqref{eq:decomp norm} and Theorem~\ref{thm:sto conv with X lp reg} the {first} statement follows. 
	{Letting $n_{\ell}=N$ for $\ell=1,\dots,L$ establishes the second statement.}
\end{proof}
\begin{remark}
	The results in this section can be generalised to non-uniform grids on each level. 
	Let
	$
	0<t_{1,\ell}<\cdots<t_{n_{\ell},\ell}=\TMAX
	$ 
	be the temporal grids that satisfies the following: 
	Letting $\delta^{\max}_{\ell}:=\max_{i=1,\dotsc,n_{\ell}}\{t_{i,\ell}-t_{i-1,\ell}\}$,
	$\delta^{\min}_{\ell}:=\min_{i=1,\dotsc,n_{\ell}}\{t_{i,\ell}-t_{i-1,\ell}\}$, we have a constant 
	$c_{\mathrm{disc}}\ge1$ such that 
	$
	\delta^{\max}_{\ell}/\delta^{\min}_{\ell}\le c_{\mathrm{disc}}
	$ holds. 
	Then, the statement of Lemma~\ref{lem:reslent get 1/lambda} can be replaced by
	\begin{align*}
	\sum_{t_{i,\ell}\le\tau_{\eta}\le\tau_{N}}
	|\mathfrak{R}_j(t_{i-1,\ell},\tau_{\eta})|^2
	(\tau_{\eta} - \tau_{\eta-1})\le
	\frac{2c_{\mathrm{disc}}}
	{\lambda_j},
	\end{align*}
	and that of Theorem~\ref{thm:sto conv with X lp reg} by
	\begin{align}
	&\sum_{\eta=1}^{N}
	\mathbb{E}\Big[
	\norm{
		\dstoconv{ B(\cdot, \what{X}^{J,L}(\cdot)) }(\tau_{\eta})
	}_{  D(A^{^{\iota+1/2}}) }^{2}
	\Big]
	(\tau_{\eta} - \tau_{\eta-1}) \nonumber\\
	&\le
	{2c_{\mathrm{disc}}}
	\mathbb{E}\bigg[
	\sumlall
	\sum_{i=1}^{n_{\ell}}
	\norm{
		\mathscr{P}_J
		B(
		t_{i-1,\ell},\what{X}^{J,L}(t_{i-1,\ell})
		)
		\mathscr{P}_L
		\sqrt{ q_\ell}h_{{\ell}}
	}_{ D(A^{\iota})}^2
	(t_{i,\ell} - t_{i-1,\ell})
	\bigg].\nonumber
	\end{align}
\end{remark}
\section{Conclusion}\label{maximal-sec:conclusion}
In this paper, we considered an implicit Euler--Maruyama scheme for a class of stochastic partial differential equations with a non-uniform time discretisation. 
For this scheme, we showed that a discrete analogue of the maximal $L^2$-regularity holds, which has the same form as the maximal regularity of the original problem.
\bibliographystyle{amsplain}
\providecommand{\bysame}{\leavevmode\hbox to3em{\hrulefill}\thinspace}
\providecommand{\MR}{\relax\ifhmode\unskip\space\fi MR }
\providecommand{\MRhref}[2]{%
	\href{http://www.ams.org/mathscinet-getitem?mr=#1}{#2}
}
\providecommand{\href}[2]{#2}

\Addresses
\end{document}